\documentclass[twoside,11pt,reqno]{amsart}
\usepackage{amsmath,amssymb,amscd,mathrsfs,amscd, todonotes,appendix}
\usepackage{graphics,verbatim}
\usepackage{todonotes}
\usepackage{enumitem}
\usepackage{hyperref}
\usepackage{setspace} 

\newcommand{\losemi}{{\otimes \kern -.78em \ltimes}}
\newcommand{\rosemi}{{\otimes \kern -.78em \rtimes}}

\newcommand{\Hom}{\ensuremath{\operatorname{Hom}}}

\newcommand{\sgn}{\operatorname{sgn}}

\newcommand{\ga}{\gamma}

\newcommand{\si}{\sigma}

\newcommand{\la}{\lambda}

\newcommand{\St}{\operatorname{St}}

\newcommand{\soc}{\operatorname{soc}}

\newcommand{\fm}{{\mathcal F}_m}

\newcommand{\fn}{{\mathcal F}_n}
\newcommand{\gn}{{\mathcal G}_n}
\newcommand{\gm}{{\mathcal G}_m}

\newcommand{\fmn}{{\mathcal F}^m_n}
\newcommand{\fdn}{{\mathcal F}^d_n}
\newcommand{\fdm}{{\mathcal F}^d_m}
\newcommand{\gmn}{{\mathcal G}^m_n}
\newcommand{\gdn}{{\mathcal G}^d_n}
\newcommand{\gdm}{{\mathcal G}^d_m}

\makeatletter
\newcommand{\leqnomode}{\tagsleft@true}
\newcommand{\reqnomode}{\tagsleft@false}
\makeatother

\newtheorem{theorem}{Theorem}[subsection]

\makeatletter\let\c@fact\c@theorem\makeatother

\makeatletter\let\c@note\c@theorem\makeatother

\newtheorem{lemma}{Lemma}[subsection]
\makeatletter\let\c@lemma\c@theorem\makeatother

\makeatletter\let\c@lemma\c@theorem\makeatother

\makeatletter\let\c@alg\c@theorem\makeatother

\newtheorem{prop}{Proposition}[subsection]
\makeatletter\let\c@prop\c@theorem\makeatother

\makeatletter\let\c@conj\c@theorem\makeatother

\newtheorem{cor}{Corollary}[subsection]
\makeatletter\let\c@cor\c@theorem\makeatother

\newtheorem{defn}{Definition}[subsection]
\makeatletter\let\c@defn\c@theorem\makeatother

\theoremstyle{definition}

\newtheorem{remark}{Remark}[subsection]
\makeatletter\let\c@remark\c@theorem\makeatother

\makeatletter\let\c@example\c@theorem\makeatother
\numberwithin{equation}{subsection}

\usepackage[capitalise]{cleveref}
\crefname{theorem}{Theorem}{Theorems}
\crefname{fact}{Fact}{Facts}
\crefname{note}{Note}{Notes}
\crefname{lemma}{Lemma}{Lemmas}
\crefname{alg}{Algorithm}{Algorithms}
\crefname{remark}{Remark}{Remarks}
\crefname{example}{Example}{Examples}
\crefname{prop}{Proposition}{Propositions}
\crefname{conj}{Conjecture}{Conjectures}
\crefname{cor}{Corollary}{Corollaries}
\crefname{defn}{Definition}{Definitions}
\crefname{equation}{\!\!}{\!\!} %Remove spacing around phantom equation name

\newcounter{listequation}

\begin{document}

\title{Restricting Rational Modules to Frobenius Kernels}

\author{\sc Christopher P. Bendel}
\address
{Department of Mathematics, Statistics and Computer Science\\
University of
Wisconsin-Stout \\
Menomonie\\ WI~54751, USA}
\thanks{Research of the first author was supported in part by an AMS-Simons Research Enhancement Grant for PUI Faculty}
\email{bendelc@uwstout.edu}

\author{\sc Daniel K. Nakano}
\address
{Department of Mathematics\\ University of Georgia \\
Athens\\ GA~30602, USA}
\thanks{Research of the second author was supported in part by
NSF grant DMS-2101941}
\email{nakano@math.uga.edu}

\author{\sc Cornelius Pillen}
\address{Department of Mathematics and Statistics \\ University
of South
Alabama\\
Mobile\\ AL~36688, USA}
%\thanks{Research of the third author was supported in part by Simons Foundation Collaboration Grant 245236}
\email{pillen@southalabama.edu}

\author{Paul Sobaje}
\address{Department of Mathematical Sciences \\
          Georgia Southern University\\
          Statesboro, GA~30458, USA}
\email{psobaje@georgiasouthern.edu}

\begin{abstract}
Let $G$ be a connected reductive group over an algebraically closed field of characteristic $p>0$.  Given an indecomposable G-module $M$, one can ask when it remains indecomposable upon restriction to the Frobenius kernel $G_r$, and when its $G_r$-socle is simple (the latter being a strictly stronger condition than the former).  In this paper, we investigate these questions for $G$ having an irreducible root system of type A.  Using Schur functors and inverse Schur functors as our primary tools, we develop new methods of attacking these problems, and in the process obtain new results about classes of Weyl modules, induced modules, and tilting modules that remain indecomposable over $G_r$.
\end{abstract}

\maketitle

\section{Introduction}

\subsection{} Let $G$ be a reductive algebraic group scheme over a field of characteristic $p>0$, and $G_r$ be a the scheme theoretic kernel of the $r$th iteration of the Frobenius map. For $r=1$, 
$G_{1}$-representations correspond to restricted representations of the restricted Lie algebra ${\mathfrak g}=\text{Lie }G$. Let $\text{Mod}(G)$ be the category of rational representations for $G$, 
and $\text{Mod}(G_{r})$ be the category of representations for $G_{r}$. In this paper, we will be interested in the following questions: 
\begin{itemize} 
\item[(i)] Given an indecomposable module $M$ in $\text{Mod}(G)$, when is $M$ indecomposable upon restriction to $G_{r}$? 
\item[(ii)] Given $M\in \text{Mod}(G)$, when is the $G_{r}$-socle, $\text{soc}_{G_{r}}M$,  of $M$ simple?
\end{itemize} 

When $M=L(\lambda)$, $\lambda\in X_{+}$ (the set of dominant weights), is a simple $G$-module, then the answers to (i) and (ii) are given by work of Curtis and Steinberg (see \cite[II Chapter 3]{rags}). For arbitrary rational $G$-modules, 
these questions can be quite difficult to answer. The goal of this paper is to indicate that answers to (i) and (ii) are tractable for various classes of modules such as induced (costandard), Weyl (standard), and indecomposable tilting $G$-modules. 

\subsection{} We will make further restrictions in this paper by working with reductive groups $G$ with underlying root system of type $A$. The main reason for this is the fundamental connection between 
polynomial representations of ${GL}_{n}(k)$ (for a field $k$) and the representations of the symmetric group $\Sigma_{d}$. In this situation, the polynomial representations of $GL_n$ are equivalent to the modules for the 
Schur algebras $S(n,d)$ where $d\geq 0$. 

Further connections have been established by Doty, Nakano, and Peters \cite{DNP96}, where they developed a polynomial representation theory for $G_rT$ ($T$ being a maximal torus) and infinitesimal Schur algebras $S(n,d)_{r}$. When $n\geq d$, one has a Schur functor in both situations to modules for the group algebra of the symmetric group $k\Sigma_{d}$. We capitalize on this setting with the Schur and inverse Schur functor, in addition to facts about symmetric group representations,  to establish that for $p\geq 3$ and $n\geq d$, the induced (costandard), Weyl (standard) and indecomposable tilting $G$-modules all remain indecomposable upon restriction to $G_{r}$ (see Corollary~\ref{C:GrIndec}). 

\subsection{}\label{SS:n<d} The paper will also focus on representations for the Schur algebra $S(n,d)$ when $n<d$. In particular, we will be interested in cases when we 
can compute the $G$-socle of the indecomposable tilting module $T(\lambda)$. We note that for $n<d$, one can use the Andersen-Haboush isomorphisms (cf. \cite[II Prop. 3.19, E.9 Lemma]{rags}) to demonstrate that certain tilting and Weyl modules are not indecomposable upon restriction to $G_rT$.

A theory of functors arising from idempotents for a finite-dimensional algebra is described in \cite{DEN}. For positive integers $m \geq n$, there exists an idempotent functor from $\text{Mod}(S(m,d))$ to $\text{Mod}(S(n,d))$, first defined by Donkin, and an inverse Schur functor ${\mathcal G}_{n}^{m}(-)$. 
 This functor is 
especially useful in the case when $m\geq d >n$. Knowledge about the modules produced under the image of this inverse Schur functor are key to understanding the polynomial representations for $\text{GL}_{n}$. 
Results about the image of ${\mathcal G}_{n}^{m}(-)$ on simple and induced modules are discussed (cf. Section~\ref{S:Gmn} and the Appendix). 

\subsection{} A specific case of question (ii) involves Donkin's Tilting Module Conjecture (TMC). The TMC is equivalent to the statement that, for a $p^r$-restricted weight $\la$, the tilting module 
$T(2(p^r-1)\rho + w_0\lambda )$ is indecomposable as a $G_r$-module, where $\rho$ denotes the Weyl weight and $w_0$ the longest Weyl group element. The reader is referred to \cite[Sections 1.1,  2.2]{BNPS23} for a detailed discussion about the history of the TMC and its related conjectures. In 
\cite[Theorem 1.2.1]{BNPS22}, the authors have shown that there are counterexamples to the TMC for all irreducible root systems except for Types $A_{n}$ and $B_{2}$. 
For $B_{2}$ and $A_{n}$ for $n\leq 3$, the TMC holds for all primes \cite[Theorem 1.1.1]{BNPS21}. The authors strongly believe that the TMC should hold for Type $A_{n}$ in general for all primes. 

From the results described in Section~\ref{SS:n<d}, one can consider the functor ${\mathcal G}_{n}^{m}(-)$ from $\text{Mod}(S(n,d))$ to $\text{Mod}(S(m,d))$. Our investigation 
led us to formulate a natural statement involving certain composition factor multiplicities (using the Mullineux map) of the image of this functor on simple modules for $\text{Mod}(S(n,d))$. For a precise formulation the reader is referred to Theorem~\ref{T:DequivG} and Corollary~\ref{C:TMCequivG}. The striking fact about this statement is that it is equivalent to the TMC for Type $A_n$ and provides more evidence that the TMC will hold for all primes in this case. 

\subsection{Notation}\label{S:notation} In this paper we will generally follow the standard conventions in \cite{rags}.  Let $k$ be an algebraically closed field of characterisic $p >0$, and let $G$ be a connected reductive algebraic group scheme over $k$ that is defined over the prime subfield ${\mathbb F}_{p}$.  Most of the results in the paper hold for an arbitrary prime, but there are some results in Sections 2 and 3 where we need to assume $p \geq 3$ (due to the nature of certain representations for a symmetric group at the  prime 2).
Usually $G$ will be either the general linear group $GL_n$ or the special linear group $SL_n$, but we work in general for now with specifications made at various points in the paper.   Given a maximal torus $T$, let $\Phi$ be the root system associated with $(G,T)$, and $\mathbb E$ be the Euclidean space associated with $\Phi$ with the inner product on ${\mathbb E} $ denoted by $\langle\ , \ \rangle$.   
Moreover, let $\Phi^+$ the set of positive roots and $\Phi^{-}$ be the corresponding set of negative roots. 
The set of simple roots determined by $\Phi^+$ is $\Delta=\{\alpha_1,\dots,\alpha_{l}\}$. 
Set $\rho$ to be the half-sum of all simple roots and $\alpha_{0}$ to be the highest short root.
Let $\alpha^{\vee}$ to be the coroot corresponding to $\alpha\in \Phi$. 

Let $B$ be the Borel subgroup given by the set of negative roots and let $U$ be
the unipotent radical of $B$. The Weyl group 
associated to $\Phi$ will be denoted by $W$, and let $w_0$ denote the longest word of $W$.
Let $X:=X(T)$ be the integral weight lattice spanned by the fundamental weights $\{\omega_1,\dots,\omega_l\}$, $X_{+}$ be the dominant weights for $G$, and $X_{r}$ be the $p^{r}$-restricted weights. 

Let $F$ denote the Frobenius morphism on $G$. The $r$th Frobenius kernel will be denoted by $G_{r}$, and its graded version by $G_{r}T$.  For $\lambda\in X_{+}$, there are four fundamental families of $G$-modules (each having highest weight $\lambda$): $L(\lambda)$ (simple), $\nabla(\lambda)$ (costandard/induced), $\Delta(\lambda)$ (standard/Weyl), and $T(\lambda)$ (indecomposable tilting). For $\lambda\in X_{r}$, let $Q_{r}(\lambda)$ denote the $G_{r}$-projective cover (equivalently, injective hull) of $L(\lambda)$ as a $G_{r}$-module. For $\lambda\in X$, if $\widehat{L}_{r}(\lambda)$ is the corresponding simple $G_{r}T$-module, let $\widehat{Q}_{r}(\lambda)$ denote the $G_{r}T$-projective cover (equivalently, injective hull) of $\widehat{L}_{r}(\lambda)$.   The $r$th Steinberg module is defined as $\text{St}_r = L((p^r-1)\rho)$.  For 
$\lambda\in X_{r}$, set $\hat{\lambda}=2(p^{r}-1)\rho+w_{0}\lambda$. 

Let $\tau:G \rightarrow G$ be the Chevalley antiautomorphism of $G$ that is the identity morphism when restricted to $T$ (see \cite[II 1.16]{rags}).  Given a finite dimensional $G$-module $M$ over $k$, the module $^{\tau}M$ is $M^*$ (the ordinary $k$-linear dual of $M$) as a $k$-vector space, with action $g.f(m)=f(\tau(g).m)$.  This defines a functor from $\text{Mod}(G)$ to $\text{Mod}(G)$ that preserves the character of $M$.  In particular, it is the identity functor on all simple and tilting modules. 

% The Coxeter number associated to $\Phi$ is $h=\langle \rho,\alpha_{0}^{\vee} \rangle +1$. 

\subsection{Acknowledgements} The authors acknowledge Rudolf Tange for useful discussions pertaining Premet's work \cite{Pr}, and to Karin Erdmann for providing references to results by Donkin and De Visscher \cite{DDV}.

%%%
%Schur Functor Background
%%%

\section{Schur Algebras and Functors}

\subsection{Schur Algebras} For positive integers $n, d$, consider the associated (finite-dimensional) Schur algebra $S(n,d) := \text{End}_{\Sigma_{d}}(V^{\otimes d})$, where $\Sigma_{d}$ is the symmetric group on $d$ letters and $V$ is the 
$n$-dimensional natural representation for $GL_{n}$. The category of $S(n,d)$-modules is equivalent to the category of polynomial representations of $GL_n$ of degree $d$.  Associated to $S(n,d)$ is the weight poset $\Lambda^+(n,d)$ of dominant partitions of $d$ into at most $n$ parts.   For each $\la \in \Lambda^+(n,d)$, one has the simple module $L(\la)$, Weyl module $\Delta(\la)$, costandard module $\nabla(\la)$, and indecomposable tilting module $T(\la)$, following our earlier conventions.  Let $P(\la)$ denote the projective cover of $L(\la)$ and $I(\la)$ denote the injective hull of $L(\la)$ as $S(n,d)$-modules.

One can also study the polynomial representation theory of $G_rT$ for $G$ being $GL_n$.
In the 1990s, Doty, Nakano and Peters \cite{DNP96} created a module category for a monoid scheme $M_{r}D$ 
whose representation theory is equivalent to a finite-dimensional 
algebra $S(n,d)_{r}$ called an {\em infinitesimal Schur algebra}.  Here, $M=M_{n}$ is the reductive monoid of 
$n\times n$ matrices, with subscheme $D$ of diagonal matrices. The monoid scheme $M_{r}D$ is a natural object of study 
if viewed as an object in the following commutative diagram 
of $k$-functors. 
$$
\CD 
T @>>> G_{r}T @>>> G \\
@VVV @VVV @VVV \\
D @>>> M_{r}D @>>> M.
\endCD
$$

We state the following result from \cite[Prop. 2.1]{DNP97} that provides the connection between restricting 
modules from $M_rD$ to $G_rT$. 

\begin{prop} \label{P:MrDtoGrT} Let $N,\ N^{\prime}$ be $M_{r}D$-modules.
\begin{itemize}
\item[$(a)$] Any decomposition $N\cong N_{1}\oplus N_{2}$ in 
$\operatorname{Mod}(G_{r}T)$ is a decomposition in 
$\operatorname{Mod}(M_{r}D)$. 
\item[$(b)$] If $N$ is an indecomposable $M_{r}D$-module, then 
$N$ remains indecomposable upon restriction to $G_{r}T$. 
\item[$(c)$] If $N\cong N^{\prime}$ in 
$\operatorname{Mod}(G_{r}T)$, then $N\cong N^{\prime}$ 
in $\operatorname{Mod}(M_{r}D)$. 
\end{itemize}
\end{prop}

\subsection{Relating Schur algebras}\label{S:Lowering}  For this section and the following section, we refer the reader to the work of Doty, Erdmann and Nakano \cite{DEN} where they describe the procedure 
of producing functors using idempotents and the general theory of these functors. 

Given positive integers $m \geq n$, there exists an idempotent $f \in S(m,d)$ such that $fS(m,d)f \cong S(n,d)$.  Multiplication by $f$ defines an exact functor $\fmn$ from $\text{Mod}(S(m,d))$ 
to  $\text{Mod}(S(n,d))$.   More precisely, for any $S(m,d)$-module $M$,
\begin{equation}\label{E:Fequiv}
\fmn(M) := fM \cong \Hom_{S(m,d)}(S(m,d)f,M) \cong fS(m,d)\otimes_{S(m,d)}M.
\end{equation}
Let $W$ (respectively, $V$) denote the natural $m$-dimensional (respectively, $n$-dimensional) module for $S(m,d)$ (respectively, $S(n,d)$).  That is, the natural $GL_m$-module (respectively $GL_n$-module).  Note that $f$ may (and will) be chosen so that $fW^{\otimes d} \cong V^{\otimes d}$.  

The behavior of the aforementioned highest weight modules under this functor is understood (see for example \cite{E}):

\begin{lemma}\label{L:fstandard} Given $\la \in \Lambda^+(m,d)$, let $\bar{M}(\la)$ denote either $\bar{L}(\la)$, $\bar{\Delta}(\la)$, $\bar{\nabla}(\la)$, or $\bar{T}(\la)$ over $S(m,d)$.  If $\la \in \Lambda^+(n,d)$, then let $M(\la)$ denote the corresponding module over $S(n,d)$.  Then 
$$
f\bar{M}(\la) = 
\begin{cases}
M(\la) &\text{ if } \la \text{ has at most } n \text{ parts, i.e., if } \la \in \Lambda^+(n,d),\\
0 &\text{ else.}
\end{cases}
$$
\end{lemma}

The functor $\fmn$ admits a right adjoint $\gmn$ from $S(n,d)$-modules to $S(m,d)$-modules defined by 
$$
\gmn(M) := \Hom_{S(n,d)}(fS(m,d),M).
$$
The functor $\gmn$ is also a right inverse to $\fmn$ (cf. \cite{A, DEN}). That is, for an $S(n,d)$-module $M$, $\fmn(\gmn(M)) \cong M$.  In particular, given a simple $S(n,d)$-module $L(\la)$, 
$$\fmn(\gmn(L(\la)) \cong L(\la).$$ This implies that $\bar{L}(\la)$ appears precisely once in a composition series for $\gmn(L(\la))$ and, for any other composition factor $\bar{L}(\mu)$ of $\gmn(L(\la))$, $\mu$ must have {\em more} than $n$ parts. The structure of $\gmn(L(\la))$ will be discussed further in Section \ref{S:Gmn}.

\subsection{Symmetric groups and the Schur functor} \label{S:sym} For $m \geq d$, there is a well-known association between $S(m,d)$-modules and $k\Sigma_d$-modules given by the (exact) Schur functor $\fm$.  As above, there is an idempotent $e \in S(m,d)$ with $eS(m,d)e \cong k\Sigma_d$ and $S(m,d)e \cong W^{\otimes d}$, where $W$ denotes the natural $GL_m$-module as above. The functor $\fm$ is given by the idempotent, that is, $\fm(M) = eM$ for an $S(m,d)$-module $M$.   Moreover, the idempotent $e$ may be chosen so that $e \in S(m,d)_r$ and $eS(m,d)_re \cong k\Sigma_d$. Thus, one may also consider the Schur functor $\fm$ from $\text{Mod}(S(m,d)_r)$ to $\text{Mod}(k\Sigma_d)$.

For a partition $\la \in \Lambda^+(m,d)$, let $\la'$ denote the transpose (or conjugate) partition.  A partition $\la$ is said to be $p$-regular if it does not have $p$ or more consecutive non-zero terms that are equal, and is otherwise known as $p$-singular. A partition $\la = (\la_1, \la_2, \dots, \la_m)$ is said to be $p$-restricted if $0 \leq \la_i - \la_{i+1} < p$ for $1 \leq i \leq m-1$ and $0 \leq \la_m < p$.   Observe that $\la$ is $p$-restricted if and only if $\la'$ is $p$-regular.  

The action of $\fm$ on simple modules is known.  Quite generally, for a finite-dimensional algebra $A$ and idempotent $e\in A$, if $L$ is a simple $A$-module, then $eL$ is either $0$ or a simple $eAe$-module. Moreover, the non-zero modules of the form $eL$ where $L$ is a simple $A$-module form a 
complete set of non-isomorphic simple $eAe$-modules.  Let $\la \in \Lambda^{+}(m,d)$,  then $eL(\la) \neq 0$ if and only if $\la$ is $p$-restricted.  For $\la$ being $p$-restricted,  set $D_{\lambda}=eL(\lambda)$.  Then the set 
$$\{D_{\lambda}\mid \text{$\lambda\in\Lambda^+(m,d)$ is $p$-restricted}\}$$ 
is a complete set of non-isomorphic simple $k\Sigma_d$-modules.

 For $p \geq 3,$ let $\text{sgn}$ denote the one dimensional sign representation of $k\Sigma_d$. Then another complete set of non-isomorphic simple $k\Sigma_d$-modules is given by 
$$\{D^{\lambda}\mid \text{$\lambda\in\Lambda^+(m,d)$ is $p$-regular}\},$$
where 
$$D^{\lambda}=D_{\lambda^{\prime}}\otimes \text{sgn}.$$ 

For $\lambda\in \Lambda^{+}(m,d)$, the Specht module over $k\Sigma_d$ is denoted by $S^{\lambda}$. 
Let $M^{\lambda}=\text{ind}_{k\Sigma^{\lambda}}^{k\Sigma_{d}} k$ be the permutation module obtained from inducing up the trivial module $k$ from the Young subgroup $\Sigma^{\lambda}$. 
The Young module $Y^{\lambda}$ is the unique indecomposable summand of $M^{\lambda}$ which contains 
$S^{\lambda}$ as a submodule (cf. \cite[Def 4.6.1]{Mar}).  

Under the Schur functor, $\fm$, we can realize these aforementioned $k\Sigma_{d}$-modules as images of $S(m,d)$-modules as 
follows (cf. \cite[$\S$ 6]{Gr} and \cite[3.5, 3.6]{Don}):  
\begin{eqnarray*} 
\fm(\nabla(\lambda))&=&S^{\lambda};\\
\fm(\Delta(\lambda))&=&(S^\lambda)^*\cong S^{\lambda^{\prime}}\otimes \text{sgn};\\
\fm(I(\lambda))&=&\fm(P(\lambda))=Y^{\lambda};\\
\fm(T(\lambda))&=&Y^{\lambda^{\prime}}\otimes \text{sgn}.
\end{eqnarray*}

\subsection{Inverse Schur functors} As in Section \ref{S:Lowering}, the Schur functor may alternately be identified as
$$
\fm(M) := eM \cong \Hom_{S(m,d)}(S(m,d)e,M) \cong eS(m,d)\otimes_{S(m,d)}M.
$$
As before, $e$ may be chosen so that $S(m,d)e \cong W^{\otimes d} \cong eS(m,d)$ (cf. \cite{E}).

\begin{defn}  This allows one to define a ``Schur functor'' $\fn$: $\operatorname{Mod}(S(n,d))$ to $\operatorname{Mod}(k\Sigma_d)$  for arbitrary $n$, (i.e., even when $n < d$) by defining 
$$
\fn(M) := V^{\otimes d}\otimes_{S(n,d)}M,
$$
where $V$ is the natural $GL_n$-module.
\end{defn}

Note that $\fn$ need not be exact (just right exact) when $n < d$.  Similar to the lowering functor $\fmn$ of Section \ref{S:Lowering}, this general Schur functor $\fn$ admits a right adjoint $\gn$ taking $\text{Mod}(k\Sigma_d)$ to $\text{Mod}(S(n,d))$ which can be defined in the following way: for a $k\Sigma_d$-module, set $\gn(M) := \Hom_{k\Sigma_d}(V^{\otimes d}, M)$.  

Again, from the more general setting of algebras and idempotents (cf. \cite{A, DEN}), when $m \geq d$, $\gm$ is a right inverse to $\fm$, although not a 2-sided inverse in general.  That is, $\fm\circ\gm(M) \cong M$.
 
However, it was shown by Doty and Nakano \cite[Theorem 6.2]{DN} (following a result of Cline, Parshall, and Scott \cite[Theorem 5.2.4]{CPS}) for $m \geq d$ and $\la \in \Lambda^+(m,d)$ that tilting (and other) $S(n,d)$-modules may be recovered from $k\Sigma_d$-modules.

\begin{theorem}\label{T:GYT1} Let $\lambda\in \Lambda(m,d)_{+}$ with $m \geq d$. Then 
\begin{itemize} 
\item[(a)] $T(\lambda)= \gm(Y^{\lambda^{\prime}}\otimes \operatorname{sgn})$ when $p\geq 3$;
\item[(b)] $\Delta(\lambda)= \gm(S^{\lambda^{\prime}}\otimes \operatorname{sgn})$ when $p\geq 3$; 
\item[(c)] $P(\lambda)= \gm(Y^{\lambda})$ when $p\geq 2$. 
\end{itemize} 
\end{theorem}

Using the ideas of Section \ref{S:Lowering}, we may extend the first two parts of this to arbitrary $n$.

\begin{theorem}\label{T:GYT} Assume $p \geq 3$.  Given positive integers $n, d$ and $\la \in \Lambda^+(n,d)$, there are isomorphisms of $S(n,d)$-modules
\begin{itemize} 
\item[(a)] $T(\lambda)= \gn(Y^{\lambda^{\prime}}\otimes \operatorname{sgn})$;
\item[(b)] $\Delta(\lambda)= \gn(S^{\lambda^{\prime}}\otimes \operatorname{sgn})$.
\end{itemize} 
\end{theorem}

\begin{proof} If $n \geq d$, this is simply Theorem \ref{T:GYT1}.  Suppose $n < d$.  Choose an integer $m \geq d$ (hence, $m > n$).  Let $W$ and $V$ be $m$- and $n$-dimensional vector spaces over $k$, respectively. Choose an idempotent $f$ as in Section \ref{S:Lowering} so that $fW^{\otimes d} = V^{\otimes d}$.  We first make the following observation.  Let $N$ be any finite-dimensional $k\Sigma_d$-module. Then
\begin{align*}
f\gm(N) &\cong \Hom_{S(m,d)}(S(m,d)f, \gm(N)) \text{ by } \eqref{E:Fequiv} \\
	&= \Hom_{S(m,d)}(S(m,d)f, \Hom_{k\Sigma_d}(W^{\otimes d},N))\\
	&\cong \Hom_{k\Sigma_d}(S(m,d)f\otimes_{S(m,d)} W^{\otimes d}, N) \text{ by Adjoint Associativity}\\
	&\cong \Hom_{k\Sigma_d}(S(m,d)\otimes_{S(m,d)}fW^{\otimes d},N)\\
	&\cong \Hom_{k\Sigma_d}(S(m,d)\otimes_{S(m,d)} V^{\otimes d},N)\\
	&\cong \Hom_{k\Sigma_d}(V^{\otimes d},N)\\
	&= \gn(N).
\end{align*}
For part (a), applying this to $N := Y^{\la'}\otimes\rm{sgn}$ and using Theorem \ref{T:GYT1} and Lemma \ref{L:fstandard}, we get
$$
\gn\left(Y^{\la'}\otimes\rm{sgn}\right) \cong f\gm\left(Y^{\la'}\otimes\rm{sgn}\right) \cong f\bar{T}(\la) \cong T(\la).
$$
Part (b) is similar.
\end{proof}

\subsection{Simple Modules}  As noted above, for $m \geq d$, the action of the Schur functor on simple modules is known (either resulting in zero or a simple $k\Sigma_d$-module).   The following shows how the inverse lowering functor $\gmn$ may be used to understand the action of $\fn$ on simples more generally.

\begin{prop}\label{P:fn} Assume $p \geq 2$. Let $\si \in \Lambda^+(n,d)$.  For any $m \geq n$ with $m \geq d$, 
$$
\fn(L(\si)) \cong e\gmn(L(\si)) = \fm(\gmn(L(\si))).
$$
\end{prop} 

\begin{proof} Given $L(\si)$, we have
\begin{align*}
\fn(L(\si)) &\cong V^{\otimes d}\otimes_{S(n,d)}L(\si)\\
	&\cong V^{\otimes d}\otimes_{fS(m,d)f}L(\si)\\
	&\cong W^{\otimes d}f \otimes_{fS(m,d)f}L(\si)\\
	&\cong W^{\otimes d}\otimes_{fS(m,d)f}L(\si)\\
	&\cong eS(m,d)\otimes_{fS(m,d)f}L(\si)\\
	&\cong e[S(m,d)\otimes_{fS(m,d)f}L(\si)]\\
	&\cong e\gmn(L(\si))\\
	&\cong \fm(\gmn(L(\si))).
\end{align*}
\end{proof}

\begin{cor}\label{C:fnD} Assume $p \geq 2$ and  $\si \in \Lambda^+(n,d)$.  Let $D_{\mu}$ be a $k\Sigma_d$-composition factor of $\fn(L(\si))$.  
Then $\mu$ is $p$-restricted.  
\begin{itemize}
\item[(i)] If $\mu = \si$, then $\si$ is $p$-restricted and $[\fn(L(\si)) : D_{\mu}] = 1$.
\item[(ii)] If $\mu \neq \si$, then $\mu$ has {\em more} than $n$ parts.
\end{itemize}
\end{cor}

\begin{proof} Choose (as needed) $m \geq n$ with $m \geq d$.     From Proposition~\ref{P:fn}, $D_{\mu}$ is a composition factor of $\fm(\gmn(L(\si)))$.   That means, $\bar{L}(\mu)$ is an $S(m,d)$-module with $\fm(\bar{L}(\mu)) = D_{\mu}$.  As noted in Section~\ref{S:sym}, $\mu$ must be $p$-restricted.  Since $\fm(\gmn(L(\si))) \cong L(\si)$ and $\fm$ is an exact functor, $[\gmn(L(\si)) : L(\si)]  = 1$, from which (i) follows.   Moreover, if $\mu \neq \si$, then $\fm(\bar{L}(\mu)) = 0$.  From Lemma~\ref{L:fstandard}, $\mu$ must have more than $n$ parts.
\end{proof}

%%%%%
%Indecomposability of modules over $G_r$ for $n \geg d$
%%%%%

\section{Restricting modules}

\subsection{} For $G = GL_n$, the inverse Schur functor can be used with information about symmetric group modules to prove the indecomposability of polynomial $G$-modules upon restriction to $G_{r}T$. Recall that the Schur functor $\fm: \text{Mod}(S(m,d)) \to \text{Mod}(k\Sigma_d)$ may be considered as a functor from $\text{Mod}(S(m,d)_r)$ to $\text{Mod}(k\Sigma_d)$.

\begin{theorem} \label{T:G-indecomposability} Assume $p \geq 2$. Let $G=GL_{m}$ and $m \geq d$. Let $N$ be an indecomposable module for $k\Sigma_{d}$. Then 
\begin{itemize} 
\item[(a)]  $\gm(N)$ is an indecomposable (polynomial) module for $G$. 
\item[(b)]  $\gm(N)|_{G_{r}T}$ is an indecomposable (polynomial) module for $G_{r}T$. 
\end{itemize}  
\end{theorem} 

\begin{proof} Note that it suffices to prove (b).  By Proposition~\ref{P:MrDtoGrT}, it suffices to show that $\gm(N)|_{G_{r}T}$ is an indecomposable $M_{r}D$-module or equivalently an indecomposable $S(m,d)_{r}$-module. 

Suppose that  $\gm(N)|_{G_{r}T} =M_{1}\oplus M_{2}$ as an $S(m,d)_{r}$-module. Then 
$$N\cong \fm (\gn(N)|_{G_{r}T}) = \fm(M_{1})\oplus \fm(M_{2}).$$ 
Now let $L_{j}$ be a simple module in the $S(m,d)_{r}$-socle of $M_{j}$ for $j=1,2$. Then 
$$0\neq \text{Hom}_{S(m,d)_{r}}(L_{j}, \gm(N)|_{G_{r}T})\cong  \text{Hom}_{k\Sigma_{d}}(\fm(L_{j}),N).$$ 
This proves that $\fm(L_{j})\neq 0$ for $j=1,2$, thus, by left-exactness of $\fm$,  $\fm(M_{j})\neq 0$ for $j=1,2$. This contradicts the indecomposability of $N$ as a $k\Sigma_{d}$-module. 
\end{proof}

\subsection{} We can now prove that Weyl modules and tilting modules for $G=GL_m$ and for weights in $\Lambda^{+}(m,d)$ where $m\geq d$ are indecomposable upon restriction to $G_rT$. 

\begin{cor}\label{C:GrIndec} Let $G=GL_{m}$. Let  $\lambda\in \Lambda^{+}(m,d)$ for 
$m\geq d$, and let $X(\lambda)$ be any of the following modules 
\begin{itemize} 
\item[(a)] $\nabla(\lambda)$ (costandard module) for $p\geq 3$; 
\item[(b)] $\Delta(\lambda)$ (standard/Weyl module) for $p\geq 3$;
\item[(c)] $T(\lambda)$ (tilting module) for $p\geq 3$;
\item[(d)] $P(\lambda)$ (projective imodule) for $p\geq 2$;
\item[(e)]  $I(\lambda)$ (injective module) for $p\geq 2$. 
\end{itemize} 
Then $X(\lambda)\mid_{G_{r}T}$ is indecomposable. 
\end{cor} 

\begin{proof} First note that the Specht module $S^{\lambda}$ is an indecomposable $k\Sigma_{d}$-module for $p\geq 3$ and the Young module $Y^{\lambda}$ is by definition indecomposable. Recall also the $\tau$-functor from Section~\ref{S:notation}.

(a) and (b): Since ${\nabla}(\lambda)=\Delta(\lambda)^{\tau}$, it suffice to prove the statement for $\Delta(\lambda)$. This follows 
by Theorem~\ref{T:GYT1} and Theorem~\ref{T:G-indecomposability} with $N=S^{\lambda^{\prime}}\otimes \text{sgn}$. (c): This follows by the same argument using Theorem~\ref{T:GYT1} and 
Theorem~\ref{T:G-indecomposability} with $N=Y^{\lambda^{\prime}}\otimes \text{sgn}$. 

(d) and (e): We can use the $\tau$ duality again. One has $P(\lambda)=I(\lambda)^{\tau}$; it suffice to prove the statement for $P(\lambda)$. 
Now apply Theorem~\ref{T:GYT1} and Theorem~\ref{T:G-indecomposability} with $N=Y^{\lambda}$. 
\end{proof}

 \subsection{Counterexamples for indecomposability when $m$ is less than $d$} \label{SS:NonIndEx} Corollary~\ref{C:GrIndec} does not hold all the time when $m<d$. One has the Andersen-Haboush isomorphism 
 (cf. \cite[II Prop. 3.19]{rags})
 $$
 \nabla((p^r - 1)\rho + p^r\gamma) \cong \St_r\otimes\nabla(\gamma)^{(r)}.
 $$
 This shows that this induced module decomposes into a direct sum of copies of $\St_r$ over $G_rT$.  A similar example holds for Weyl modules.  
 
 For a simple example, let $p = 3$, $r = 1$, and consider the $SL_3$-weight $\gamma = \omega_1$. Then $(p-1)\rho + 3\gamma = 5\omega_1 + 2\omega_2$.  Converting this to a partition gives 
 $(7,2,0)$, a 3-part partition of 9. In the above context, one has $m=3$ and $d=9$, where clearly $m < d$. 
 
Now suppose that $\mu \in (p^{r}-1)\rho + X_{r}$ and $T(\mu)$ is indecomposable as a $G_{r}T$-module. Then, for $\ga \in X_+$, by \cite[E.9 Lemma]{rags}
 $$T(\mu+p^{r}\gamma)\cong T(\mu)\otimes T(\gamma)^{(r)}.$$
For $\ga \neq 0$, this generates more counterexamples that do not involve the Steinberg module.   This formula is relevant when Donkin's Tilting Module Conjecture holds (see Section \ref{S:TMC}), namely when $\mu=2(p^{r}-1)\rho+w_{0}\lambda$ for $\lambda\in X_{r}$.

 \subsection{Premet's Theorem} Premet looked at the question of indecomposablity for Weyl modules for $G=SL_{2}$ when restricted to $G_{1}$. This can be applied to the induced 
 modules $\nabla(\lambda)$ where $\lambda\in X_{+}$. The following theorem summarizes 
 this situation in the rank one case by using \cite[Lemma 1.1]{Pr} and the analysis in Section~\ref{SS:NonIndEx}. 
  
 \begin{theorem} Let $G=SL_{2}$ and $\lambda\in X_{+}$. Then $\nabla(\lambda)$ restricted to $G_{1}T$ is 
 indecomposable if and only if one of the conditions hold:
 \begin{itemize} 
 \item[(a)] $p\nmid \lambda+1$,
 \item[(b)] $\lambda=p-1$. 
\end{itemize} 
\end{theorem} 

As a corollary one can state an indecomposability criterion for standard modules for $S(2,d)$. 

\begin{cor} Let $S(2,d)$ be the Schur algebra for $GL_{2}$ and $\la = (\la_1,\la_2)$ be a partition of $d$. Then  $\nabla(\la)$ is indecomposable over $G_{1}T$ if and only if 
one of the following conditions hold:
\begin{itemize} 
 \item[(a)] $p\nmid (\lambda_{1}-\lambda_{2})+1$,
 \item[(b)] $\lambda_{1}-\lambda_{2}=p-1$.
\end{itemize} 
\end{cor}

\subsection{} In general, an interesting problem would be to determine for $G$ reductive when the module $\nabla(\lambda)$ is indecompoasble over $G_{r}$ where $\lambda\in X_{+}$. 
For $\lambda\in X_{+}$, one has $\lambda=\lambda_{0}+p^{r}\lambda_{1}$ where $\lambda_{0}\in X_{r}$ and $\lambda_{1}\in X_{+}$. 
Observe that 
$$\text{Hom}_{G_{r}}(L(\mu),\nabla(\lambda))\cong \text{ind}_{B/B_{r}}^{G/G_{r}}\ \text{Hom}_{B_{r}}(L(\mu),\lambda_{0}+p^{r}\lambda_{1}).$$ 
If $\mu\in X_{r}$ and $\lambda\in X_{r}$, then $\lambda_{1}=0$. This implies that $\text{soc}_{G_{r}}\nabla(\lambda)=L(\lambda)$ because the $B_{r}$-head of $L(\mu)$ is $\mu$. We can now make the following statement. 

\begin{prop} Let $G$ be a reductive algebraic group and $\lambda\in X_{r}$. Then $\nabla(\lambda)$ is indecomposable as a $G_{r}$-module. 
\end{prop} 

A natural question to ask is whether $\nabla(\lambda)$ is indecomposable if and only if $\lambda\in X_{r}$ or $\lambda$ is not in $(p^{r}-1)\rho+p^{r}X_{+}$? This would provide a natural generalization of 
Premet's result for $SL_{2}$.

%%%%
%Socles
%%%%

\section{Socles of Tilting Modules}

\subsection{The Mullineux Correspondence}  Given a simple $k\Sigma_d$-module, tensoring with the sign representation yields another simple $k\Sigma_d$-module whose highest weight is given by the Mullineux map.  The Mullineux map, denoted $M_p(-)$, is an involution on the set of $p$-regular partitions of $d$ as defined for example in \cite[\S 4.2]{Mar}.   The connection with simple modules may be seen in one of two ways.  Given a $p$-regular partition $\si$, 
\begin{equation}\label{E:Mupper}
D^{\si}\otimes\sgn \cong D^{M_p(\si)}.
\end{equation}
Alternately, given a $p$-restricted partition $\si$,
\begin{equation}\label{E:Mlower}
D_{\si}\otimes\sgn \cong D_{M_p(\si')'}.
\end{equation}
Note that when the prime $p=2$, the Mullineux map is simply the identity.  With this in mind, the results in this section hold for an arbitrary prime $p$.

\subsection{Socles}  Let $\mu \in \Lambda^+(n,d)$ be a $p$-regular partition.  Our goal is to understand the $S(n,d)$-socle of $T(\mu)$.  In particular, one would like to understand when the socle is simple.  For $n \geq d$,  Kouwenhoven \cite{K} and Donkin \cite{D} showed that the socle of $\Delta(\mu)$ is $L(M_p(\mu)')$.  In Corollary \ref{C:socle}, we extend that result to some weights when $n < d$.  More generally, Theorem \ref{T:socle2} provides a general formula for the socle of $T(\mu)$ in terms of the compositions factors of of the modules $\gmn(L(\si))$.   To that end, we further investigate the functor $\gmn(-)$ in the next subsection.

\subsection{The $\gmn$-map}\label{S:Gmn}  In order to provide further information on the socle of $T(\mu)$, we need additional information the structure of   $\gmn(L(\si))$ for the map $\gmn$ from $\text{Mod}(S(n,d))$ 
to $\text{Mod}(S(m,d))$ where $m > n$.  Let $\bar{I}(\si)$ the injective hull of $\bar{L}(\si)$ in $\text{Mod}(S(m,d))$-mod. From \cite[Corollary 3.3]{DEN}, $\gmn(L(\si))$ may be identified as the largest submodule of $\bar{I}(\si)$ whose socle quotient has composition factors $\bar{L}(\mu)$ with $\fmn(\bar{L}(\mu)) = 0$. That is, as noted earlier, for any composition factor $\bar{L}(\mu)$ of $\gmn(L(\si))$with $\mu \neq \si$, the partition $\mu$ must have more than $n$-parts.  The following lemma illustrates a basic fact about partitions.

\begin{lemma}\label{L:parts} Let $\mu, \si \in \Lambda^+(d)$ with $\si$ having at most $n$-parts, i.e., $\si \in \Lambda^+(n,d)$. Assume that $\mu \geq \si$. Then $\mu$ has at most $n$-parts.
\end{lemma}

\begin{proof} Write $\si = (\si_1,\si_2,\dots,\si_n)$ and $\mu = (\mu_1, \mu_2, \dots)$.   By the inequality assumption,
$$
\sum_{i = 1}^n\mu_i \geq \sum_{i = 1}^n\si_i = d.
$$
Since $\mu$ is a partition of $d$, this necessarily means $\mu_i = 0$ for $i > n$. In other words, $\mu$ has at most $n$ parts.
\end{proof}

The injective $S(m,d)$-module $\bar{I}(\si)$ admits a good filtration. That is, there exists a filtration
$$
0 = F_0 \subseteq F_1 \subseteq F_2 \subseteq \cdots \subseteq F_r = \bar{I}(\si),
$$
where, for each $1 \leq i \leq r$, $F_{i}/F_{i-1} \cong \nabla(\mu)$ for some partition $\mu$ (depending on $i$).  
From the general principle that
$$
[\bar{I}(\si) : \bar{\nabla}(\mu)] = [\bar{\nabla}(\mu) : \bar{L}(\si)].
$$
one concludes that $\mu \geq \si$ (for each $\mu$ appearing) and $\bar{\nabla}(\si)$ appears precisely once in the filtration (at the bottom).  In other words, $F_1/F_0 \cong \bar{\nabla}(\si)$ and $\mu > \si$ for $2 \leq i < r$.  We can now bound the image $\gmn(L(\si))$ inside a costandard module. 

\begin{prop}\label{P:GonL} Assume $m > n$ and $\si \in \Lambda^+(n,d)$. Then $\gmn(L(\si)) \subseteq \bar{\nabla}(\si)$.
\end{prop}

\begin{proof} For convenience, set $M := \gmn(L(\si))$.  Since $M \subseteq \bar{I}(\si)$, the above good filtration on $\bar{I}(\si)$ induces a filtration on $M$:
$$
0 = M\cap F_0 \subseteq M\cap F_1 \subseteq \cdots \subseteq M\cap F_r = M.
$$
For $i \geq 2$, consider the embedding $(M\cap F_i)/(M\cap F_{i-1}) \hookrightarrow F_i/F_{i-1} \cong \bar{\nabla}(\mu)$, where $\mu > \si$. Note that $\mu$ depends on $i$.  Then
$$
\soc_{G} \left((M\cap F_i)/(M\cap F_{i-1})\right) \subseteq \soc_{G} \left(\bar{\nabla}(\mu)\right) = \bar{L}(\mu).
$$
If the lefthand socle is non-zero, then it must be $\bar{L}(\mu)$, meaning that $\bar{L}(\mu)$ must be a composition factor of $M\cap F_i$ and hence of $M$.   However, from Lemma \ref{L:parts}, $\mu$ has at most $n$-parts. As $\mu$ is not $\si$, $\bar{L}(\mu)$ cannot be a composition factor of $M$. Therefore, the socle and, hence, the module must be zero. That is $(M\cap F_i)/(M\cap F_{i-1}) = 0$ or $M\cap F_i = M\cap F_{i-1}$ for $2 \leq i \leq r$.  Thus $M = M\cap F_r = M\cap F_1$.   In other words, $M \subseteq F_1 \cong \bar{\nabla}(\si)$.
\end{proof}

For $\si \in \Lambda^+(n,d)$, it follows from a dual version of \cite[Thm. 4.5]{DEN} that $\gmn(\nabla(\si)) = \bar{\nabla}(\si)$.   Left exactness of $\gmn(-)$ immediately gives Proposition \ref{P:GonL}.

We make one last observation about the $\gmn$-functor.   In the following statement, the modules on the left side of the equality are $S(m,d)$-modules and those on the right are $S(d,d)$-modules. The subscripts on the $\bar{L}$-modules are used to explicitly denote this distinction.  

\begin{lemma}\label{L:comp} Given partitions $\si \in \Lambda^+(n,d)$ and $\nu \in \Lambda^+(m,d)$ with $n \leq m \leq d$, there is an equality of composition multiplicities:
$$
\left[\gmn(L(\si)) : \bar{L}_m(\nu)\right] = \left[\gdn(L(\si)) : \bar{L}_d(\nu)\right].
$$
\end{lemma}

\begin{proof} By standard properties of Levi subgroups, using similar subscript notation, one has  
$$
\left[\bar{\nabla}_m(\si): \bar{L}_{m}(\nu)\right] = \left[\bar{\nabla}_d(\si) : \bar{L}_d(\nu)\right].
$$
From Proposition \ref{P:GonL}, $\gmn(L(\si)) \subseteq \bar{\nabla}_m(\si)$ and $\gdn(L(\si)) \subseteq \bar{\nabla}_d(\si)$. Recall from Lemma \ref{L:fstandard} that $\fdm(\bar{\nabla}_d(\si)) = \bar{\nabla}_m(\si)$.   If $\bar{L}_d(\nu)$ appears in $\gdn({\nabla}_d(\si))$, then there is a composition series for $\bar{\nabla}_d(\si)$ with highest weights (other than $\si$) lying below $\bar{L}_d(\nu)$ having more than $n$ parts.   Applying $\fdm$ gives a composition series for $\bar{\nabla}_m(\si)$ with the same property, and hence, $\bar{L}_m(\nu)$ appears in $\gmn(L(\si))$. Conversely, since $\gdn(L(\si)) = \gdm\circ\gmn(L(\si))$, if $\bar{L}_m(\nu)$ appears in $\gmn(L(\si))$, that factor will ``survive'' upon lifting further to $\gdn(L(\si))$.
\end{proof}

If one knows the module structure of $\bar{\nabla}(L(\si))$, one can compute the module $\gmn(L(\si))$.   The necessary module structure may be obtained (in small ranks) with the use of GAP \cite{GAP} and the Weyl module package written by S. Doty \cite{Doty}.    In looking at examples, $\gmn(L(\si))$ generally seems to be significantly smaller than $\bar{\nabla}(\si)$. Indeed, it is often simply $\bar{L}(\si)$.  Although it can be all of $\bar{\nabla}(\si)$.   This is illustrated in the Appendix \ref{S:appendix}, where the $S(4,12)$-module structure of $\bar{\nabla}(\si)$ is presented for all 3-part partitions $\si$ of 12 with $p = 3$. The values of the functor 
$\mathcal{G}^4_3 : \text{Mod}(S(3,12)) \to \text{Mod}(S(4,12))$ (i.e., from $SL_3$-modules to $SL_4$-modules) on $L(\si)$ are identified.

\subsection{The socle in general}\label{S:socgen} We now provide a general method to determine the socle of an indecomposable $S(n,d)$-tilting module with $p$-regular highest weight.  To that end, we first recall that certain tilting modules can be identified as injective hulls (equivalently, projective covers). The lemma below is a slight reformulation of \cite[Lemma 3.3]{DDV}.

\begin{lemma}[Donkin-De Visscher]\label{L:inj} Assume $p \geq 2$.  Let $\mu$ be a $p$-regular partition in $\Lambda^+(d,d)$ for some $d$. As an $S(d,d)$-module, $T(\mu)$ is the injective hull of $L(M_p(\mu)')$.  
\end{lemma}

We can now provide information on the number of times an irreducible module appears in the socle of a tilting module. 

\begin{theorem}\label{T:socle2} Let $\mu \in \Lambda^+(n,d)$ be $p$-regular. Set $m$ to be the number of parts of $M_p(\mu)'$.   Given $\si \in \Lambda^+(n,d)$, the number of times that $L(\si)$ appears in the $S(n,d)$-socle of $T(\mu)$ is 
$$
\left[\gmn(L(\si)) : \bar{L}(M_p(\mu)')\right].
$$
\end{theorem}

\begin{proof} Consider an $S(n,d)$-module $L(\si)$.  Then
\begin{align*}
\Hom_{S(n,d)}\left(L(\si),T(\mu)\right)
	&= \Hom_{S(n,d)}\left(L(\si),\fdn(\bar{T}(\mu)\right) \text{ by \ref{L:fstandard}}\\
	&\cong \Hom_{S(d,d)}\left(\gdn(L(\si)),\bar{T}(\mu)\right) \text{ by adjointness}\\
	&\cong \Hom_{S(d,d)}\left(\gdn(L(\si)),I(M_p(\mu)')\right) \text{ by Lemma \ref{L:inj}}.
\end{align*}
Hence, the number of times that $L(\si)$ appears in the $S(n,d)$-socle of $T(\mu)$ is 
$$
\left[\gdn(L(\si)) : \bar{L}(M_p(\mu)')\right].
$$
The claim follows from Lemma \ref{L:comp}.
\end{proof}

The following gives us a criterion for the socle of $T(\mu)$ to be simple.

\begin{cor}\label{C:socle} Let $\mu \in \Lambda^+(n,d)$ be a $p$-regular partition and assume that $M_p(\mu)' \in \Lambda^+(n,d)$ (i.e., has at most $n$ parts).  Then the $S(n,d)$-socle of $T(\mu)$ is $L(M_p(\mu)')$.  
\end{cor}

\begin{proof} Consider $m$ and $n$ as in Theorem \ref{T:socle2}.  In this scenario, we have $m = n$.   Thus $\gmn(L(\si)) = \mathcal{G}_n^n(L(\si)) = L(\si)$ and $\bar{L}(M_p(\mu)') = L(M_p(\mu)')$.   Hence, the only $L(\si)$ appearing in the socle is $\si = M_p(\mu)'$ and precisely once.
\end{proof}

Recall that the Mullineux map is an involution on the set of $p$-regular partitions and conjugation gives a one-to-one correspondence between $p$-regular weights and $p$-restricted weights.   As such, Corollary \ref{C:socle} not only gives a condition on the simplicity of the socle of $T(\mu)$ but also tells us that the highest weight of that simple module is $p$-restricted.

\subsection{Examples} The following examples illustrate Theorem \ref{T:socle2} and Corollary \ref{C:socle} over a Schur algebra, as well as how this may be applied to $SL_n$.  

\subsubsection{} As an application of Corollary \ref{C:socle} in a situation with $n < d$, consider the algebra $S(3,5)$ with $p = 3$ and the partition $(5) \in \Lambda^+(3,5)$.   Then $M_3((5))' = (3,2)' = (2,2,1)$ has 3 parts.  Corollary~\ref{C:socle} says that as an $S(3,5)$-module (or, more generally, for an $S(n,5)$-module with $n \geq 3$), the socle of $T(5)$ is $L(2,2,1)$.  As an $SL_3$-module (i.e., type $A_2$), this is equivalent to the socle of $T(5\omega_1)$ being $L(\omega_{2})$.   

While the corollary does not apply to $S(2,5)$ (or $SL_2)$, note that the module $T(5)$ over $S(2,5)$ (or $T(5\omega_1)$ over $SL_2$) also has simple socle, since $T(5)$ is in fact simple.  One could also apply Theorem \ref{T:socle2} to compute the $S(2,5)$-socle of $T(5)$.

\subsubsection{} This next example demonstrates how Theorem \ref{T:socle2} may be used to conclude simplicity of a socle. Let $p = 3$ and consider the partition $(8,4)$ as a 3-part partition of 12.   Observe that $M_3((8,4))' = (4,4,2,2)$.  Since this has 4 parts, Corollary \ref{T:socle2} does not apply to $T(8,4)$ over $S(3,12)$.   However,  from Appendix \ref{S:appendix}, one finds that $\si = (4,4,4)$ is the only $\si \in \Lambda^+(3,12)$ for which $\bar{L}(4,4,2,2)$ appears as a composition factor of $\mathcal{G}^4_3(L(\si))$.   From Theorem \ref{T:socle2}, we may conclude that the $S(3,12)$-socle of $T(8,4)$ is $L(4,4,4)$.   Equivalently, this is the statement that the $SL_3$-socle of $T(4\omega_1 + 4\omega_2)$ is $k$; a fact known to hold since Donkin's Tilting Module Conjecture (see Section \ref{S:TMC}) is valid for $SL_3$.

\subsubsection{} The socle of a tilting module need not always be simple, as will be demonstrated in this example.  Let $p = 3$ and consider the tilting module $T(3,2,1)$ over $S(4,6)$.  Then $M_3((3,2,1))' = (5,1)' = (2,1,1,1,1)$ which has 5 parts, and so Corollary~\ref{C:socle} does not apply over $S(4,6)$ (or $S(n,6)$ with $n \leq 4$).  Over $S(5,6)$ (or any $S(m,6)$ with $m \geq 5$), Corollary \ref{C:socle} does apply, saying that the socle of $T(3,2,1)$ is $L(2,1,1,1,1)$.   Translating to $SL_5$ (i.e., type $A_4$), this says the $SL_5$-socle of $T(\omega_{1}+\omega_{2}+\omega_{3})$ is $L(\omega_{1})$.

Considering $T(3,2,1)$ as an $S(4,6)$-module, we may apply Theorem \ref{T:socle2}.
To do so, we must investigate the structure of $\mathcal{G}_4^5(L(\si))$ for all $\si \in \Lambda^+(4,6)$.  Using GAP \cite{GAP} and S. Doty's Weyl module program \cite{Doty}, one finds that $L(2,1,1,1,1)$ appears (once) in $\mathcal{G}_4^5(L(2,2,2))$ and (once) in $\mathcal{G}_4^5(L(3,1,1,1)) = \bar{\nabla}(3,1,1,1)$.   Hence, the $S(4,6)$-socle of $T(3,2,1)$ is $L(2,2,2)\oplus L(3,1,1,1)$. In particular, it is not simple.  
This is equivalent to saying that the $SL_4$-socle (i.e., type $A_3$) of $T(\omega_{1}+\omega_{2}+\omega_{3})$ is $L(2\omega_{3})\oplus L(2\omega_{1})$.  Indeed, one could apply the Jantzen Sum Formula to directly show that the Weyl module $\Delta(\omega_{1}+\omega_{2}+\omega_{3})$ has socle $L(2\omega_{1})\oplus L(2\omega_{3})$. 

If one considers the tilting module $T(3,2,1)$ in $\text{Mod}(S(3,6))$, one finds that the socle is in fact simple.   As above, we may apply Theorem \ref{T:socle2}.   Again using GAP \cite{GAP},  one may compute $\mathcal{G}_3^5(L(\si))$ for all $\si \in \Lambda^+(3,6)$.  One finds that $\bar{L}(2,1,1,1,1)$ appears as a composition factor only in $\mathfrak{G}_3^5(L(2,2,2)) = \bar{\nabla}(2,2,2)$ (and precisely once).  Hence, the $S(3,6)$-socle of $T(3,2,1)$ must be $L(2,2,2)$.  Over $SL_3$, this is equivalent to saying that the $SL_3$-socle (i.e., type $A_2$) of $T(\omega_{1}+\omega_{2})$ is $k$.

%%%%
%%Applications
%%%%

\section{Applications}

Returning to the representation theory of $SL_n$, in this section we provide two significant applications of our preceding results: one to the simplicity of certain Weyl modules and one to Donkin's Tilting Module Conjecture. We begin with some observations on weights and partitions that will be needed to apply our main results.  The results of this section again hold for arbitrary primes.

\subsection{Connecting weights and partitions} We first recall that the definitions of a $p$-restricted weight and a $p$-restricted partition are interchangeable.

\begin{lemma}\label{L:pres} Let $\la \in X_1$.  When considered as a partition, $\la$ is $p$-restricted.
\end{lemma} 

\begin{proof}  Let $\la = \sum_{i=1}^{n-1}\la_i \omega_{i}$ (expressed in terms of fundamental weights in type $A_{n-1}$), then the associated partition (with $n$ parts) is $\mu = (\la_1 + \cdots + \la_{n-1}, \la_1 + \cdots + \la_{n-2},\dots,\la_1 + \la_2,\la_1,0)$.  For each $1 \leq i \leq n-1$, $\mu_i - \mu_{i-1} = \la_i < p$ by assumption. Moreover, $\mu_n = 0 < p$.
\end{proof} 

Next, in order to apply Theorem \ref{T:socle2} and Corollary \ref{C:socle}, we need to know that certain weights are $p$-regular.

\begin{lemma}\label{L:preg} Let $\la \in X_1$.  As a partition, $\hat{\la}$ is strictly decreasing.  Hence, it is $p$-regular for any $p$.
\end{lemma}

\begin{proof} As a partition with $n$ parts, $\rho = (n-1, \dots, 2, 1, 0)$. Treating $\la$ as a partition, write $\la = (\la_1, \dots, \la_n)$. By assumption and the previous lemma, $\la$ is $p$-restricted.  For convenience, let $\mu$ denote the weight $\hat{\la}$ as a partition. Then
$$
\mu = 2(p-1)\rho + w_0\la = (2(p-1)(n-1) + \la_n, 2(p-1)(n-2) + \la_{n-1}, \dots, 2(p-1) + \la_2,\la_1).
$$
For each $1 \leq i \leq n-1$, 
\begin{align*}
\mu_i - \mu_{i+1} &= 2(p-1)(n-i) + \la_{n-i+1} - (2(p-1)(n-i-1) + \la_{n-i}) \\
	&= 2(p-1) + \la_{n - i + 1} - \la_{n-i}\\ 
	&= p - 2 + [p - (\la_{n-i} - \la_{n-i+1})] \\
	& > p - 2 \text{ since } \la \text{ is } p\text{-restricted}\\
	&\geq 0.
\end{align*}
Hence, $\mu_i > \mu_{i+1}$ as claimed.
\end{proof} 

\subsection{Partial Steinberg modules}\label{S:partial}  Over any $G$, the first Steinberg module $\St_1 = L((p-1)\rho)$ is well known to be tilting. That is, 
$$\St_1 = L((p-1)\rho) = \nabla((p-1)\rho) = \Delta((p-1)\rho) = T((p-1)\rho).$$   Consider what one might call a ``partial Steinberg'' module  $\Delta((p-1)\omega_1 + \cdots + (p-1)\omega_i)$ for $1 \leq i \leq n-1$.  Over $SL_n$, one may use Corollary \ref{C:socle} to show that such modules (and some slight generalizations thereof) are always simple (and hence tilting).

We begin with some observations on the nature of Mullineux operation for certain partitions.  Given a rational number $q$, let $\lceil q \rceil$ denote the least integer  greater than or equal to $q$.

\begin{prop}\label{P:mulllength}  Let $\la =(\la_1, \la_2, \dots , \la_n)$ be a partition of $d$ satisfying 
\begin{itemize}
\item[(a)] $\la_n > 0,$
\item[(b)] $\la_i - \la_{i+1} \geq p-1$ for $i=1, \dots , n-1.$
\end{itemize}
Then $M_p(\la)'$ is a $p$-restricted partition of $d$ of length $\displaystyle{ \lceil \la_1/(p-1) \rceil.}$  Moreover, assume that $\la_n \leq p-1$ and $\la_i - \la_{i+1} = p-1$ for $i = 1, \dots, n-1$.  Then $M_p(\la)' = \la$.
\end{prop}

\begin{proof} Observe that such a partition is certainly $p$-regular.  We work with Xu's algorithm (cf.~\cite{Xu}) as described in \cite[Section 6]{BK}. Note that they work with $p$-restricted partitions. We therefore apply their algorithm to the conjugate weight $\la'.$  Condition (2) implies that $\la'$ is $p$-restricted. The $i$th column of $\la'$ contains $\la_i$ dots. Since $\la_i -\la_{i+1} \geq p-1$ the $p$-segments which make up the $p$-rim consist of the bottom $p$ dots in each column, except possibly in the most right column, where the $p$-segment could contain the entire column. At each step of Xu's algorithm, we are removing the $p$-rim with the exception of the $p$th node of each $p$-segment. This results in removing exactly $p-1$ dots from each column, except possibly the most right one, where all the dots could be removed. 

In any case, after the removal, the resulting partition is the conjugate of $(\la_1-(p-1), ... ,\la_{n-1}-(p-1), \max\{\la_n -(p-1),0\}).$  If the last entry is zero we remove it. In any case we obtain a new partition which still satisfies conditions (a) and (b). The process terminates after the first column vanishes which is exactly after $\displaystyle{ \lceil \la_1/(p-1) \rceil}$ many steps.

Consider now the special case. Then $\la_1 = (n-1)(p-1) + \la_n$ and so $\displaystyle{ \lceil \la_1/(p-1) \rceil} = n$. From Xu's algorithm, $M_p(\la) = (\mu_1, \mu_2, \dots, \mu_n)$, where $\mu_i$ is the number of nodes removed at the $i$th step.  At the first step, precisely $\la_n$ nodes will be removed in the last column and so a total of $(n-1)(p - 1) + \la_n$ nodes will be removed.  That is, $\la_1$ nodes will be removed.   Similarly, one sees that $(n-2)(p-2) + \la_{n-1} = \la_2$ nodes will be removed at the second step. Continuing in this manner, one sees that $M_p(\la)' = (\la_1,\la_2,\dots,\la_n) = \la$.
\end{proof}

From Proposition~\ref{P:mulllength} and Corollary \ref{C:socle} we obtain immediately:

\begin{cor} Let $\la =(\la_1, \la_2, \dots , \la_n)$ be a partition of $d$ satisfying 
\begin{itemize}
\item[(a)] $\la_n > 0,$
\item[(b)] $\la_i - \la_{i+1} \geq p-1$ for $i=1, \dots, n-1.$
\end{itemize}
Set $\displaystyle{s= \lceil \la_1/(p-1)\rceil}.$
Then, for $m \geq s,$ the $S(m,d)$-socle of the tilting module $T (\la)$ is $L(M_p(\la)'),$ a simple $p$-restricted module. Moreover, assume that $\la_n \leq p-1$ and $\la_i - \la_{i+1} = p-1$ for $i = 1, \dots, n-1$.  Then $T(\la) \cong L(\la)$.
\end{cor}

The following is a translation of this corollary into the algebraic group setting.

\begin{cor}
Assume  $G = SL_{n+1}$ and $\la=\sum_{i=1}^{n}\la_i \omega_{i} \in X_1$ with $\la_n > 0$.  Set $t = \langle \la, \alpha_0^{\vee} \rangle$ and 
$\displaystyle{s= \lceil t/(p-1)\rceil+n - 1}.$
For $ m \geq s,$ let $\mu = \sum_{i=1}^{n}\mu_i \omega_{i}$ be the $\text{SL}_{m+1}$-weight
defined by
$$
\mu_i = \begin{cases}
p - 1 +\la_i & \mbox{ for } 1 \leq i \leq n-1, \\
\la_n & \mbox{ for } i = n,\\
0 & \mbox{ for } n+1 \leq i \leq m.
\end{cases} 
$$
Then the $\text{SL}_{m+1}$-modules  $T(\mu)$ and $\Delta(\mu)$ will have a simple socle with $p$-restricted highest weight.
Moreover, if $\la_1 = \la_2 = \cdots = \la_{n-1} = 0$, then $T(\mu) \cong \Delta(\mu) \cong L(\mu).$
\end{cor} 

\begin{proof} Note that $\displaystyle{t = \langle \la, \alpha_0^{\vee} \rangle = \sum_{i = 1}^n\la_i}$.  Consider the partition $\si = (\si_1, \si_2, \dots, \si_n)$ defined by $\displaystyle{\si_i = \sum_{j=i}^n\mu_j}$.  Then $T(\mu)$ as an $SL_{m+1}$-module for $m \geq n$ is equivalent to considering $T(\si)$ as an $S(m+1,d)$-module (for appropriate $d$).   The claims follow from the preceding corollary noting that the $s$ therein is given by $s= \lceil \mu_1/(p-1)\rceil = n - 1 + \lceil t/(p-1)\rceil$.
\end{proof}

For type $A$, using Proposition \ref{P:SocInd} (below), one also obtains an equivalent to the  well-known result \cite[II Lemma 11.10]{rags}.

\begin{cor} Assume  $G = SL_{n+1}$ and $\la \in X_1$ with $\langle \la, \alpha_0^{\vee} \rangle \leq p-1.$
Then the $G$-module $T((p-1)\rho + \la)$ is indecomposable as a $G_1T$-module, 
\end{cor}

\subsection{The Tilting Module Conjecture}\label{S:TMC}  In the general setting of a connected reductive group $G$, Donkin's Tilting Module Conjecture states that for $\la \in X_r$, the module $\widehat{Q}_r(\la)$ lifts to an indecomposable tilting $G$-module, specifically $T(\hat{\lambda})$, where $\hat{\lambda} := 2(p^r - 1)\rho + w_0\la$.  Recall also that, if the conjecture holds for $r = 1$, then it hold for all $r \geq 1$.   The goal of this section is to show that for $SL_n$ the validity of the TMC is equivalent to understanding the composition multiplicities of the image of a simple module under the inverse lowering functor $\gmn$. See Theorem \ref{T:DequivG}.   

We first observe that simplicity of the $G$-socle of $T(\hat{\la})$ gives a condition for the indecomposability of $T(\hat{\la})$ over $G_rT$, noting this is valid quite generally, i.e., not just for $SL_n$.  

\begin{prop}\label{P:SocInd} Let $G$ be a connected reductive algebraic group and $\la \in X_r$.   The $G$-socle of $T(\hat{\la})$ is simple if and only if
$$
T(\hat{\la})|_{G_rT} \cong \widehat{Q}_r(\la).
$$
\end{prop}

\begin{proof}  It follows from \cite[Proposition 4.1.3]{So18} that there is a decomposition of $G_rT$-modules
$$T(\hat{\lambda}) \cong \widehat{Q}_r(\lambda) \oplus \bigoplus_{\mu \in X_r(T), \mu \ne \lambda} \widehat{Q}_r(\mu) \otimes \Hom_{G_r}(L(\mu),T(\hat{\lambda})).$$
Since $G_r$ is normal in $G$ and $T(\hat{\lambda})$ is a $G$-module, each $\Hom_{G_r}(L(\mu),T(\hat{\lambda}))$
is a $G$-submodule of $T(\hat{\lambda})$.  Furthermore,
$$\Hom_{G_r}(L(\mu),T(\hat{\lambda})) \ne \{0\}$$
if and only if there is a $G$-submodule of $T(\hat{\lambda})$ of the form $L(\mu) \otimes L(\mu^{\prime})^{(r)}$ for some dominant weight $\mu^{\prime}$.  On the other hand, the decomposition above shows that $L(\lambda)$ is a $G$-submodule of $T(\hat{\lambda})$.  It follows that $T(\hat{\la})$ is indecomposable over $G_rT$ if and only if it has a simple $G$-socle.

%If the isomorphism holds, then the $G_rT$ socle of $T(\hat{\la})$ is necessarily the $G_rT$-socle of $Q_r(\la)$, which is $L(\la)$. In particular, the $G_rT$-socle of $T(\hat{\la})$ is simple. That immediately implies that the $G$-socle of $T(\hat{\la})$ is simple.  %Conversely, suppose that the $G_r$-socle of $T(\hat{la})$ is simple.  From \cite[Cor. 2A]{P}, $T(\hat{\la})$ is a $G$-summand of $\St_r\otimes L((p^r - 1)\rho + w_0\la)$.

\end{proof} 

Note that the nature of $\hat{\la}$ is important in this theorem. As seen in Section \ref{SS:NonIndEx}, we may have a module $T(\mu)$ with simple $G$-socle that decomposes upon restriction to $G_rT$.
Combining this result with Theorem \ref{T:socle2} gives a criterion for indecomposability of the restriction of $T(\hat{\la})$ to $G_1T$.

\begin{theorem}\label{T:DequivG} Assume $G = SL_n$.   Let $\la \in X_1$ with $\hat{\la} \in \Lambda^+(n,d)$ (when considered as a partition).  Let $m$ denote the number of parts in $M_p(\hat{\la})'$. Then $T(\hat{\la})$ is indecomposable over $G_1T$ if and only if 
$$
\left[\gmn(L(\si)) : \bar{L}(M_p(\hat{\la})')\right] = 
\begin{cases}
0 \text{ if } \si \neq \la,\\
1 \text{ if } \si = \la,
\end{cases}
$$
for all $\si \in \Lambda^+(n,d)$.
\end{theorem}

\begin{proof} From Proposition \ref{P:SocInd}, the indecomposability of $T(\hat{\la})$ over $G_1T$ is equivalent to the simplicity of the $G$-socle of $T(\hat{\la})$ or equivalently to the simplicity of the $S(n,d)$-socle of $T(\hat{\la})$.   From Lemmas \ref{L:pres} and \ref{L:preg}, the weight $\hat{\la}$ (considered as a partition) is $p$-regular.  Thus one may apply Theorem \ref{T:socle2} with $\mu = \hat{\la}$ to get the claim.  
\end{proof}

It follows from Proposition~\ref{P:mulllength} that, for $G = SL_n$ and $\la \in X_1(T),$ the length of the partition $M_p(\hat{\la})'$ is at most $2n.$ Moreover, length $2n$ happens only when $\la =0,$ where $M_p(\hat{\la})'= (n(p-1), n(p-1), (n-1)(p-1), (n-1)(p-1), ...  , p-1, p-1)$.

Theorem~\ref{T:DequivG} along with Proposition~\ref{P:SocInd} yields the following result that demonstrates that verifying the TMC is equivalent to determining composition factor multiplicities. 

\begin{cor}\label{C:TMCequivG} Assume $G = SL_n$.  The following are equivalent. 
\begin{itemize} 
\item[(a)] Donkin's Tilting Module Conjecture
\item[(b)] For all $\la \in X_1$ with $\hat{\la} \in \Lambda^+(n,d)$ (when considered as a partition), with $m$ equal to the number of parts in $M_p(\hat{\la})'$, 
$$
\left[\gmn(L(\si)) : \bar{L}(M_p(\hat{\la})')\right] = 
\begin{cases}
0 \text{ if } \si \neq \la,\\
1 \text{ if } \si = \la,
\end{cases}
$$
for all $\si \in \Lambda^+(n,d)$.
\end{itemize} 
\end{cor}

To verify the TMC for $SL_n$, one would a priori need to apply Corollary~\ref{C:TMCequivG}(b) to each $\la \in X_1$.   However, one knows that the result must hold or not hold simultaneously for the dual weights $\la$ and $-w_0\la$.   So that eliminates a number of cases. Moreover, for a given weight $\la$, when attempting to verify Corollary~\ref{C:TMCequivG}, one only needs to consider $\si$ with $\si \leq \hat{\la}$.    This still leaves a large number of computations, however the case of $SL_4$ and $p = 3$ might be tractable with the aid of a computer to compute the modules $\gmn(L(\si))$ that need to be considered.

%%%%
%%Appendix with Weyl Module Structures
%%%%

%\newpage
\appendix
\section{Induced modules}\label{S:appendix}

Let $p = 3$ and consider the Schur algebras $S(3,12)$ and $S(4,12)$ along with the functor $\mathcal{G}^4_3 : \text{Mod}(S(3,12)) \to \text{Mod}(S(4,12))$.  Presented here is the $S(4,12)$-submodule structure of the induced module $\bar{\nabla}(\si)$ for all $\si \in \Lambda^+(3,12)$, i.e., all 3-part partitions of 12. From this structure, an identification of $\mathcal{G}^4_3(L(\si))$ is made for each $\si$.  The computations are made (for $SL_4$) using the program GAP \cite{GAP} and S. Doty's Weyl module program \cite{Doty}.

\vskip.2cm\noindent
The structure of $\bar{\nabla}(12)$:
$$\bar{L}(8,2,2)$$
$$\bar{L}(8,4)$$
$$\bar{L}(11,1)\oplus \bar{L}(6,6)$$
$$\bar{L}(12).$$
Since the factors immediately above $\bar{L}(12)$ have only 2 parts ($\leq 3$), they (nor anything higher) cannot lie in $\mathcal{G}^4_3(L(12))$.    Hence, $\mathcal{G}(L(12)) = \bar{L}(12)$.

\vskip.2cm\noindent
The structure of $\bar{\nabla}(11,1)$:
$$\bar{L}(6,2,2,2)$$
$$\bar{L}(9,2,1)$$
$$\bar{L}(9,3)\oplus\bar{L}(8,2,2)\oplus\bar{L}(7,4,1)$$
$$\bar{L}(10,1,1)\oplus\bar{L}(8,4)$$
$$\bar{L}(11,1)$$
Hence, $\mathcal{G}^4_3(L(11,1)) = \bar{L}(11,1)$.

\vskip.2cm\noindent
The structure of $\bar{\nabla}(10,2)$:
$$\bar{L}(7,2,2,1)$$
$$\bar{L}(10,2).$$
Here, $(7,2,2,1)$ has more than 3 parts. Hence, $\mathcal{G}^4_3(L(10,2)) = \bar{\nabla}(10,2)$.

\vskip.2cm\noindent
The structure of $\bar{\nabla}(10,1,1)$:
$$\bar{L}(6,3,2,1)$$
$$\bar{L}(9,1,1,1)\oplus\bar{L}(6,3,3)\oplus\bar{L}(6,2,2,2)$$
$$\bar{L}(9,2,1)$$
$$\bar{L}(7,4,1)$$
$$\bar{L}(10,1,1)$$
Hence, $\mathcal{G}^4_3(L(10,1,1)) = \bar{L}(10,1,1)$.

\vskip.2cm\noindent
The structure of $\bar{\nabla}(9,3)$:
$$\bar{L}(7,3,2)$$
$$\bar{L}(8,2,2)\oplus\bar{L}(6,3,3)\oplus\bar{L}(5,5,2)\oplus\bar{L}(6,2,2,2)\oplus\bar{L}(4,4,4)$$
$$\bar{L}(7,4,1)\oplus\bar{L}(5,4,3)$$
$$\bar{L}(8,4)\oplus\bar{L}(9,2,1)$$
$$\bar{L}(9,3)$$
Hence, $\mathcal{G}^4_3(L(9,3)) = \bar{L}(9,3)$.

\vskip.2cm\noindent
The structure of $\bar{\nabla}(9,2,1)$:
$$\bar{L}(3,3,3,3)$$
$$\bar{L}(7,3,2)\oplus\bar{L}(6,3,2,1)$$
$$\bar{L}(8,2,2)\oplus\bar{L}(9,1,1,1)\oplus\bar{L}(6,3,3)\oplus\bar{L}(6,2,2,2)\oplus\bar{L}(4,4,4)$$
$$\bar{L}(7,4,1)$$
$$\bar{L}(9,2,1)$$
Hence, $\mathcal{G}^4_3(L(9,2,1)) = \bar{L}(9,2,1)$.

\vskip.2cm\noindent
The structure of $\bar{\nabla}(8,4)$:
$$\bar{L}(4,4,2,2)$$
$$\bar{L}(6,2,2,2)\oplus\bar{L}(4,4,4)\oplus\bar{L}(3,3,3,3)$$
$$\bar{L}(6,5,1)\oplus\bar{L}(7,3,2)$$
$$\bar{L}(8,2,2)\oplus\bar{L}(5,5,2)\oplus\bar{L}(6,3,3)\oplus\bar{L}(4,4,4)\oplus\bar{L}(3,3,3,3)$$
$$\bar{L}(7,4,1)\oplus\bar{L}(6,6)\oplus\bar{L}(5,4,3)$$
$$\bar{L}(8,4)$$
Hence, $\mathcal{G}^4_3(L(8,4)) = \bar{L}(8,4)$.

\vskip.2cm\noindent
The structure of $\bar{\nabla}(8,3,1)$:
$$\bar{L}(5,3,2,2)$$
$$\bar{L}(8,3,1)$$
Hence, $\mathcal{G}^4_3(L(8,3,1)) = \bar{\nabla}(8,3,1)$.

\vskip.2cm\noindent
The structure of $\bar{\nabla}(8,2,2)$:
$$\bar{L}(4,4,2,2)$$
$$\bar{L}(6,2,2,2)\oplus\bar{L}(4,4,4)\oplus\bar{L}(3,3,3,3)$$
$$\bar{L}(7,3,2)$$
$$\bar{L}(8,2,2).$$
Hence, $\mathcal{G}^4_3(L(8,2,2)) = \bar{L}(8,2,2)$.

\vskip.2cm\noindent
The structure of $\bar{\nabla}(7,5)$:
$$\bar{L}(4,4,3,1)$$
$$\bar{L}(7,2,2,1)$$
$$\bar{L}(7,5).$$
Hence, $\mathcal{G}^4_3(L(7,5)) = \bar{\nabla}(7,5)$.

\vskip.2cm\noindent
The structure of $\bar{\nabla}(7,4,1)$:
$$\bar{L}(5,4,2,1)$$
$$\bar{L}(6,3,2,1)\oplus\bar{L}(5,4,3)\oplus\bar{L}(4,4,2,2)$$
$$\bar{L}(6,4,1,1)\oplus\bar{L}(5,5,2)\oplus\bar{L}(6,3,3)\oplus\bar{L}(6,2,2,2)\oplus\bar{L}(3,3,3,3)\oplus\bar{L}(3,3,3,3)$$
$$\bar{L}(6,5,1)\oplus\bar{L}(7,3,2)$$
$$\bar{L}(4,4,4)$$
$$\bar{L}(7,4,1)$$
Hence, $\mathcal{G}^4_3(L(7,4,1)) = \bar{L}(7,4,1)$.

\vskip.2cm\noindent
The structure of $\bar{\nabla}(7,3,2)$:
$$\bar{L}(5,4,2,1)$$
$$\bar{L}(6,3,2,1)\oplus\bar{L}(5,4,3)\oplus\bar{L}(4,4,2,2)$$
$$\bar{L}(6,3,3)\oplus\bar{L}(5,5,2)\oplus\bar{L}(6,2,2,2)\oplus\bar{L}(4,4,4)\oplus\bar{L}(3,3,3,3)$$
$$\bar{L}(7,3,2)$$
Hence, the structure of $\mathcal{G}^4_3(L(7,3,2))$ is
$$\bar{L}(6,2,2,2)\oplus\bar{L}(3,3,3,3)$$
$$\bar{L}(7,3,2).$$

\vskip.2cm\noindent
The structure of $\bar{\nabla}(6,6)$:
$$\bar{L}(4,4,2,2)$$
$$\bar{L}(4,4,4)\oplus\bar{L}(6,2,2,2)\oplus\bar{L}(3,3,3,3)$$
$$\bar{L}(6,5,1)$$
$$\bar{L}(3,3,3,3)$$
$$\bar{L}(6,6)$$
Hence, the structure of $\mathcal{G}^4_3(L(6,6))$ is
$$\bar{L}(3,3,3,3)$$
$$\bar{L}(6,6).$$

\vskip.2cm\noindent
The structure of $\bar{\nabla}(6,5,1)$:
$$\bar{L}(5,4,2,1)$$
$$\bar{L}(5,4,3)\oplus\bar{L}(6,3,2,1)\oplus\bar{L}(4,4,2,2)$$
$$\bar{L}(5,5,2)\oplus\bar{L}(6,4,1,1)\oplus\bar{L}(6,3,3)\oplus\bar{L}(4,4,4)\oplus\bar{L}(6,2,2,2)\oplus\bar{L}(3,3,3,3)$$
$$\bar{L}(6,5,1)$$
Hence, the structure of $\mathcal{G}^4_3(L(6,5,1))$ is
$$\bar{L}(6,4,1,1)\oplus\bar{L}(6,2,2,2)\oplus\bar{L}(3,3,3,3)$$
$$\bar{L}(6,5,1)$$

\vskip.2cm\noindent
The structure of $\bar{\nabla}(6,4,2)$:
$$\bar{L}(6,4,2)$$
Hence, $\mathcal{G}^4_3(L(6,4,2)) = \bar{L}(6,4,2) = \bar{\nabla}(6,4,2)$.

\vskip.2cm\noindent
The structure of $\bar{\nabla}(6,3,3)$:
$$\bar{L}((5,4,2,1))\oplus\bar{L}(3,3,3,3)$$
$$\bar{L}((5,4,3))\oplus\bar{L}(6,3,2,1)$$
$$\bar{L}(6,3,3).$$
Hence, the structure of $\mathcal{G}^4_3(L(6,3,3))$ is
$$\bar{L}(6,3,2,1)$$
$$\bar{L}(6,3,3).$$

\begin{remark} The case of $\bar{\nabla}(6,3,3)$ is an interesting example where the head of the induced module (equivalently, socle of the Weyl module) is not simple.  Working over $SL_4$, the socle of the Weyl module $\Delta(3\omega_1 + 3\omega_3)$ is $L(\omega_1 + 2\omega_2 + \omega_3)\oplus k$.
\end{remark}

\vskip.2cm\noindent
The structure of $\bar{\nabla}(5,5,2)$:
$$\bar{L}(5,4,2,1)$$
$$\bar{L}(5,4,3)$$
$$\bar{L}(5,5,2)$$
Hence, $\mathcal{G}^4_3(L(5,5,2)) = \bar{L}(5,5,2)$.

\vskip.2cm\noindent
The structure of $\bar{\nabla}(5,4,3)$:
$$\bar{L}(4,4,2,2)$$
$$\bar{L}(5,4,2,1)\oplus\bar{L}(4,4,4)\oplus\bar{L}(3,3,3,3)$$
$$\bar{L}(5,4,3).$$
Hence, the structure of $\mathcal{G}^4_3(L(5,4,3))$ is
$$\bar{L}(5,4,2,1))\oplus\bar{L}(3,3,3,3)$$
$$\bar{L}(5,4,3).$$

\vskip.2cm\noindent
The structure of $\bar{\nabla}(4,4,4)$:
$$\bar{L}(4,4,2,2)$$
$$\bar{L}(4,4,4).$$
Hence, $\mathcal{G}^4_3(L(4,4,4)) = \bar{\nabla}(4,4,4)$.

%%%%%
%Bibliography
%%%%

\providecommand{\bysame}{\leavevmode\hbox
to3em{\hrulefill}\thinspace}

\end{document}